\documentclass[a4paper,11pt]{amsart}
\usepackage[utf8x]{inputenc}

\newtheorem{thm}{Theorem}

\newtheorem{prop}[thm]{Proposition}
\newtheorem{lm}[thm]{Lemma}

\newtheorem{coro}[thm]{Corollary}
\theoremstyle{definition}
\newtheorem{rmq}[thm]{Remark}
\newtheorem{ex}{Example}
\numberwithin{thm}{section}
\numberwithin{equation}{section}
\newcommand{\s}[1]{\sigma_{ext}(#1)}
\newcommand{\lh}{\mathcal{B}(\mathcal{H})}

\title[Extended spectrum of quasi-normal operators]{Extended spectrum and extended eigenspaces of quasi-normal operators}
\author{GILLES CASSIER}
\address{Universit\'e de Lyon; Universit\'e Lyon 1; Institut Camille Jordan CNRS UMR 5208; 43, boulevard du 11 Novembre 
1918, F-69622 Villeurbanne}
\email{cassier@math.univ-lyon1.fr}
\author{HASAN ALKANJO}
\email{alkanjo@math.univ-lyon1.fr}

\begin{document}

\begin{abstract}
We say that a complex number $\lambda$ is an extended eigenvalue of a bounded linear operator $T$
on a Hilbert space $\mathcal{H}$ if there exists a nonzero bounded linear operator $X$ acting on $\mathcal{H}$,
called extended eigenvector associated to $\lambda$, and satisfying the equation $TX=\lambda XT$.
In this paper we describe the sets of extended eigenvalues and extended eigenvectors for the quasi-normal operators.
\bigskip

\flushleft {2010} Mathematics Subject Classification: Primary 47B20, 47A75, 47A80. Secondary 47A25, 47A60.\\
Keywords and phrases: Extended Eigenvalues, Extended Eigenspaces, Intertwining values, 
$\lambda$-intertwining operators, normal operators, quasi-normal operators.
\end{abstract}
\maketitle
\begin{section}{Introduction And Preliminaries}
Let $\mathcal{H}$ be a separable complex Hilbert space, and denote by $\lh$ the algebra of 
all bounded linear operators on $\mathcal{H}$. If $T$ is an operator in $\lh$, then a complex number $\lambda$ 
is an extended eigenvalue of $T$ if there is a nonzero operator $X$ such that $TX=\lambda XT$. 
We denote by the symbol $\sigma_{ext}(T)$ the set of extended eigenvalues of $T$.
The subspace generated by extended eigenvectors corresponding to $\lambda$ will be denoted by $E_{ext}(T,\lambda)$. 

The concepts of extended eigenvalue and extended eigenvector are closely related with generalization of famous Lomonosov's
theorem on existence of non-trivial hyperinvariant subspace for the compact operators on a Banach space, which were 
done by S. Brown in \cite{brown}, and Kim, Moore and Pearcy in \cite{prcy}, and is stated as follows :

\textit{If an operator T on a Banach space has a non-zero compact eigenvector, then T has a nontrivial hyperinvariant
subspace.}

The special case, where $T$ commutes with a non-zero compact operator, is the original theorem of Lomonosov \cite{lmnsv}.

Extended eigenvalues and their corresponding extended eigenvectors were studied by several authors 
(see for example \cite{hsn}, \cite{volt1}, \cite{casalk}, \cite{kar} and \cite{lamb}).

In \cite{volt1}, Biswas, Lambert and Petrovic have introduced this notion and they have shown that 
$\s{V}=]0,+\infty[$ where $V$ is the well-known integral Volterra operator on the space $L^2[0,1]$. 
In \cite{kar}, Karaev gave a complete description of the set of extended eigenvectors of $V$.

Recently, in \cite{shkarin}, Shkarin has shown that there is a compact quasinilpotent 
operator $T$ for which $\s{T}=\{1\}$, 
that which allows to classify this type of operators.

In \cite{casalk}, the authors has given an accurate and practical description of the set of 
extended eigenvectors of normal operators.

In this paper we treat a more generalize class of operators, that is the quasinormal operators.

In section $2$ we introduce the sets of intertwining values of a couple of operators 
and $\lambda$-intertwining operators associated with a couple of operators and an intertwining value. We give
a complete description of the set of intertwining values associated with a quasinormal operator 
and a operator of the form $A \otimes S$ where $A$ is an injective positive operator and $S$ is the usual
forward shift on the Hardy space $H^2$. This is the main result of the paper and it is used several times
in the sequel. In particular, we apply this result in order to describe the extended spectrum of a pure quasinormal
operator.

In section $3$, Theorem \ref{subnl} gives a description of extended eigenvectors for any injective 
subnormal operator. In particular,
we give a description of extended eigenvectors related to the canonical decomposition of a subnormal operator in 
sum of normal and pure subnormal operators. 

Section $4$ is devoted to the complete description of the extended eigenvalues and the extended eigenspaces of a general
quasinormal operator.

In section $5$, we generalize a theorem of Yoshino which gives a necessary and sufficient 
condition that an operator commuting with a quasinormal operator 
have an extension commuting with the normal extension of the quasinormal operator. 
In particular, we generalize this to operators intertwining two quasinormal operators, such a result which will 
give us the relationship between extended eigenvectors of some pure quasinormal operators and 
their minimal normal extensions.

\end{section}

\begin{section}{Intertwining values and $\lambda$-intertwining operators of quasi-normal operators}
In this section, we characterize the set of extended eigenvalues of a quasi-normal operator. 
Recall that an operator $T\in\lh$ is quasi-normal if it commutes with its modulus $|T|:=(T^*T)^{1/2}$, i.e., $T|T|=|T|T$.
Furthermore, $T$ is pure if it has no reducing subspaces $\mathcal{M}\neq\{0\}$ such that $T|_{\mathcal{M}}$ is normal.
Since the normal operators have been accomplished in \cite{casalk}, we will focus in this 
section on the case of pure quasi-normal operators.
First we will show some auxiliary results.
\begin{prop}\label{sm}
 Let $T_1,T_2\in\lh$, then $\s{T_1}\s{T_2}\subset\s{T_1\otimes T_2}$, where $T_1\otimes T_2$ is the tensor product of
$T_1$ and $T_2$.
\end{prop}
\begin{proof}
Let $\lambda_i\in\s{T_i}$ and $X_i\in E_{ext}(T_i,\lambda_i)\backslash\{0\}$, $i=1,2$. If we consider $X:=X_1\otimes X_2$,
then $X$ is a nonzero operator in $E_{ext}(T_1\otimes T_2,\lambda_1\lambda_2)$, which implies 
\[
 \lambda_1\lambda_2\in\s{T_1\otimes T_2}.
\]
\end{proof}
Now, if we denote by $S$ the unilateral shift (which we suppose that it is acting on Hardy space $H^2$), then A. 
Brown proved the following theorem (see \cite{brownq}).
\begin{thm}\label{brown}
 An operator $T\in\lh$ is a pure quasinormal operator if and only if there is an injective positive operator $A$ on 
a Hilbert space $\mathfrak{L}$ such that $T$ is unitarily equivalent to $A\otimes S$, acting on $\mathfrak{L}\otimes H^2$.
\end{thm}
\begin{rmq}\label{smbl}
The two following remarks will be frequently used in the sequel.\\%\begin{enumerate}
1) Let $T=V_T|T| \in \mathcal{B}(\mathcal{H})$ be the polar decomposition of a pure quasinormal operator $T$. 
The subspace
$\mathfrak{L}_T=\mathcal{H}\ominus V_T\mathcal{H} $ is invariant by $|T|$ and we can choose the positive operator $A$ in
the above theorem by setting $A:=A_T=|T|_{\rvert}\mathfrak{L}_T$. In this case we will denote by 
$U_T \in \mathcal{B} (\mathcal{H}, \mathfrak{L}_T\otimes H^2) $ the unitary operator
such that $A_T\otimes S=U_TTU_T^{\ast}$. Proposition \ref{sm} and Theorem \ref{brown} already show that
$\s{A_T}\cdot \mathbb{D}^{c} \subseteq \s{T}$. We will frequently identify the space $\mathfrak{L}_T\otimes H^2$ with 
the space $\oplus_{k=0}^{\infty} \mathfrak{L}_T$.\\
2) Let $\mathcal{H}$ be a Hilbert space, and $T\in\lh$. Suppose that there exist $X,U\in\lh$, with 
$U$ is an invertible operator, such that $T=U^{-1}SU$. Then, one can easily verify that
$\s{T}=\s{S}$. Moreover, for all $\lambda\in\s{T}$, $E_{ext}(T,\lambda)=UE_{ext}(S,\lambda)U^{-1}$. 
%\end{enumerate}
\end{rmq}
%As a consequence, it is convenient to describe the extended spectrum and extended eigenspaces for the operator $A\otimes S$,
%where $A$ is an injective positive operator. To this purpose, 
To our purpose, we introduce the following useful sets of operators.
Let $(A,B) \in \mathcal{B}(\mathcal{H}_1)\times \mathcal{B}(\mathcal{H}_2)$ and $r\in\mathbb{R}_+^*$,
we define $ \mathcal{A}_r(A,B)$ by setting
\begin{align*}
 \mathcal{A}_r(A,B)&=\{L\in\mathcal{B}(\mathcal{H}_2,\mathcal{H}_1) : \exists c\geq0\ such\ that\ \forall x \in
 \mathcal{H}_2,\\ 
 &\forall 
 n \in \mathbb{N}, ||r^{-n}A^nLx||\leq c||B^nx||\}.
\end{align*}
When $\mathcal{H}_1=\mathcal{H}_2:=\mathcal{H}$ and $A=B$, the set $\mathcal{A}_r(A,A)$ is denoted as $\mathcal{A}_r(A)$ 
or if no confusion is possible, we write simply $\mathcal{A}_r$.
Moreover, in the case of the positive operator $A$ is invertible, the set $\mathcal{A}_{|\lambda|}$ is defined by
\[
 \mathcal{A}_{|\lambda|}=\{L\in\mathcal{B}(\mathcal{H}) : \sup_{n\in\mathbb{N}}||\lambda^{-n}A^nLA^{-n}||<+\infty\}.
\]
In addition, for $|\lambda|=1$, we get the Deddens algebra given in \cite{dedalg}.
We also define the intertwining values associated with the couple of operators $(A,B)$ by setting
$\Lambda_{int}(A,B)=\{ \lambda \in \mathbb{C}: \exists X \in \mathcal{B}(\mathcal{H}_2, \mathcal{H}_1)
\ such\ that\ AX=\lambda XB\}$. If $\lambda \in \Lambda_{int}(A,B)$, we denote by $E_{int}(A,B,\lambda)$ the space of
$\lambda$-intertwining operators $X \in \mathcal{B}(\mathcal{H}_2,\mathcal{H}_1)$, that is 
operators such that
$AX=\lambda XB$. When $A=B$, $\Lambda_{int}(A,A)$ is exactly the extended spectrum 
of the operator $A$, and if $\lambda \in \s{A}$, then the space of $\lambda$-intertwining operators 
is exactly $E_{ext}(T,\lambda)$.
The next result will be used several times in the sequel.
\begin{prop}\label{criterion}
Let $R$ be an injective quasinormal operator acting on a Hilbert space $\mathcal{H}$ and let 
$A$ be an injective positive operator on a Hilbert space $\mathfrak{L}$. 
%and $T=A\otimes S$. 
%%If we denote $m_A=\inf\{<Ax,x> : ||x||=1\}$,
Then, $\lambda \in \Lambda_{int}(R,A\otimes S)$ if and only if $\mathcal{A}_{|\lambda|}(R,A) \neq \{0\} $. 
%\[
%\s{T}\subset\mathbb{D}(0,\frac{m_A}{||A||}[^c. 
%\]
\end{prop}
\begin{proof}
Assume that $\lambda \in \Lambda_{int}(R,A\otimes S)$ and let 
$X=[X_0,\ldots,X_n,\ldots]\in \mathcal{B}(\oplus_{k=0}^{+\infty}\mathfrak{L},\mathcal{H})$ be a nonzero operator
satisfying $RX=\lambda X(A\otimes S)$. Since $R$ is injective, we have $\lambda \neq0$.
An easy calculation shows that $RX_k=\lambda X_{k+1}A$,
and hence $R^kX_0=\lambda^k X_kA^k$ for any $k \in \mathbb{N}$. On the one hand, since the range of $A^k$ is dense, 
it implies
that $X_0$ is necessarily non-null. On the other hand, we see that
\[
 |\lambda|^{-n}||R^nX_0x||\leq  |X||||A^nx||,\   \ \forall n\in\mathbb{N},\ \ \forall x\in\mathfrak{L}.
\]
Consequently, $X_0$ is a non-zero element of $\mathcal{A}_{|\lambda|}(R,A)$.

Reciprocally, if $L\in\mathcal{A}_{|\lambda|}(R,A) \setminus \{0\}$, we define for all $n\in\mathbb{N}$, the operator  
\[
\begin{array}{ccccc}
\breve{X}_n & : & ImA^n & \to & \mathcal{H} \\
 & & A^nx & \mapsto & \lambda^{-n}R^nLx. \\
\end{array} 
\]
Since $L\in\mathcal{A}_{|\lambda|}(R,A)$, there is $c\geq0$ such that for all $y\in ImA^n$, 
$||\breve{X}_ny||\leq c||y||$. Also, $ImA^n$ is dense in $\mathfrak{L}$, thus $\breve{X}_n$
has a (unique bounded) extension on $\mathfrak{L}$, which will be denoted by $X_n$.
It remains to verify that $RX_n=\lambda X_{n+1}A$, for all $n\in\mathbb{N}$. So, let $x\in\mathfrak{L}$,
and $y=A^nx$, then
\begin{align*}
 RX_ny&=RX_nA^nx=R\breve{X}_nA^nx=\lambda^{-n}R^{n+1}Lx\\
 &=\lambda\lambda^{-(n+1)}R^{n+1}Lx=\lambda\breve{X}_{n+1}A^{n+1}x=\lambda X_{n+1}Ay.
\end{align*}
By density, we get $RX_n=\lambda X_{n+1}A$, and hence $X\in E_{int}(R,A\otimes S,\lambda) \setminus \{0\}$, as we wanted.
\end{proof}
Let $T$ be a self-adjoint operator acting on a Hilbert space $H$, we denote by $m_T=\inf\{<Tx,x> : ||x||=1\}$. 
We observe that $m_T=1/||T^{-1}||$ when $T$ is an invertible positive operator. 
Also denote, as usual, by $\sigma(T)$ and $\sigma_p(T)$ the spectrum and the point spectrum of $T$ respectively.
The following Theorem is the main result of the paper and will be used several times in the sequel.

\begin{thm}\label{intertwining values}
Let $R \in \mathcal{B}(\mathcal{H})$ be an injective quasinormal operator and let $A \in \mathcal{B}(\mathfrak{L})$ 
be an injective positive operator, then we have:
\begin{enumerate}
 \item if $(m_{|R|}, ||A||) \in \sigma_p(|R|)\times \sigma_p(A)$, then 
\[
\Lambda_{int}(R,A\otimes S)=\mathbb{D}(0,\frac{m_{|R|}}{||A||}[^c;
\]
 \item if $ (m_{|R|},||A||) \notin \sigma_p(|R|)\times \sigma_p(A)$, then 
\[
 \Lambda_{int}(R,A\otimes S)=\mathbb{D}(0,\frac{m_{|R|}}{||A||}]^c.
\]
\end{enumerate}
\end{thm}
\begin{proof}
The first step consists in proving the inclusion 
\[
\mathbb{D}(0,\frac{m_{|R|}}{||A||}]^c \subseteq \Lambda_{int}(R,A\otimes S).
\]
Let $0<\varepsilon < ||A||$,
if we denote by $E^A$ (resp. $E^{|R|}$) the spectral measure of $A$ (resp. $|R|$), then 
we can choose a nonzero vector $a$ (resp. a nonzero vector $b$) in $E^{|R|}([ m_{|R|} , m_{|R|} +\varepsilon \})(\mathcal{H})$ 
(resp. in $E^A(\left[ ||A||-\varepsilon, ||A|| \right] )(\mathfrak{L})$) because 
$m_{|R|}=\inf\{\lambda: \lambda\in \sigma(|R|)\}$ 
(resp. $||A||=\sup\{\lambda : \lambda \in \sigma(A)\}$). Observe that $b$ can be written under the form $b=A^nb_n$ where
\[
 b_n=\left( \int_{||A||-\varepsilon}^{||A||}t^{-n}dE^{A}(t) \right)b.
\]
Set $L=a\otimes b$, since $R$ is quasinormal we have
\begin{align*}
 ||R^nLx||&=|||R|^na|| | <x,b> |=|||R|^na||  ||b_n|| | <A^nx,\frac{b_n}{||b_n||} > | \\
 &\leq |||R|^na||  ||b_n||||A^nx|| \leq ||a|| ||b|| \left( \frac{m_{|R|}+\varepsilon}{||A||-\varepsilon}\right)^n||A^nx||.
\end{align*}
Hence, the non-null operator $L$ belongs to 
$\in\mathcal{A}_{\frac{m_{|R|}+\varepsilon}{||A||-\varepsilon}}(R,A)$. Applying Proposition \ref{criterion}, we see that 
$\mathbb{D}(0,\frac{m_{|R|}+\varepsilon}{||A||-\varepsilon}[^c \subseteq \Lambda_{int}(R,A\otimes S)$. Since
$\varepsilon$ could be arbitrarily chosen in $]0,||A||[$, we obtained the desired inclusion. 

The second step is to prove that 
$\Lambda_{int}(R,A\otimes S) \subseteq \mathbb{D}(0,\frac{m_{|R|}}{||A||}[^c$. 
%First, observe that we can assume 
%that $m_{|R|} > 0$, otherwise the inclusion is obvious. 
Let $\lambda \in \Lambda_{int}(R,A\otimes S)$, we know from 
Proposition \ref{criterion} that there exists a non-null operator $L \in \mathcal{A}_{| \lambda |}(R,A)$.
Recall that since the operator $R$ is quasinormal, we have $|R^n|=|R|^n$.
Therefore, there exists an absolute positive constant $C$ such that
$
(m_{|R|})^{2n}L^{\ast}L \leq L^{\ast} |R|^{2n}L=L^{\ast} R^{\ast n}R^nL
\leq C^2 |\lambda|^{2n}A^{2n}.
$
Then, we necessarily have 
\[ 
\frac{m_{|R|}}{||A||} \left(\frac{||L||}{C}\right)^{\frac{1}{n}} \leq |\lambda|, 
\] 
and letting $n \rightarrow \infty$ we obtained the desired conclusion. 
We are now in position to prove the announced assertions.

(1) By hypothesis, there exists a couple of unit 
eigenvectors $(u,v) \in \mathcal{H}\times \mathfrak{L}$ such that $|R|u=m_{|R|}u$ and $Av=||A||v$. 
We can see that the 
operator $L=u\otimes v$ is in $\mathcal{A}_{m_{|R|}/||A||}(R,A)$. From 
Proposition \ref{criterion}, we deduced that the circle $\mathcal{C}(0,m_{|R|}/||A||)$ centered in $0$ and 
of radius $m_{|R|}/||A||$ is contained in $\Lambda_{int}(R,A\otimes S)$.
Using the two firsts steps of the proof, we can conclude.

(2) Suppose that $\lambda \in \Lambda_{int}(R,A\otimes S)$ with $|\lambda|=\frac{m_{|R|}}{||A||}$, 
then there exists $X\neq0$ such that $RX=\lambda X(A\otimes S)$. 
Since $R$ is injective, then $\lambda\neq 0$ and hence $|R|$ is invertible ($m_{|R|}>0$).
As in the proof of the last proposition, we write $X=[X_0,\ldots,X_n,\ldots]$, and we get
$R^nX_0=\lambda^n X_nA^n$. Let $R=V|R|$ be the polar decomposition of the operator $R$, 
since $R$ is injective we see that $V$ is an isometry. 
Choose $x\in\mathfrak{L}$ and $y=R^{\ast n}b \in Im(R^{\ast n})$, we derive that
\[
 |<X_0x,y>|=\left(\frac{m_{|R|}}{||A||}\right)^n|<X_nA^nx,b>|
  \leq ||X||||\left(\frac{A}{||A||}\right)^nx||m_{|R|}^n||b||.
\]
Since $R$ is quasinormal, the isometry $V$ commute with $|R|$.
Then, observe that we can choose $b$ in the closure of the range of $R^n$ which is contained in the range of $V^n$, hence 
we can write $b=V^nc$.
Therefore, we get $ |||R|^{-n}y||= ||V^{\ast n}b||=||V^{\ast n}V^nc||=||c||=||b||$.
Then, using the density of the range of $R^{\ast n}$, for all $(x,y)\in \mathfrak{L}\times \mathcal{H} $  we obtain
\[
 |<X_0x,y>| \leq ||X||||\left(\frac{A}{||A||}\right)^nx||||m_{|R|}^n|R|^{-n}y||.
\]
But
\[
 ||m_{|R|}^n|R|^{-n}y||^2=\int_{m_{|R|}}^{||R||}m_{|R|}^{2n}\frac{1}{t^{2n}}
 dE^{|R|}_{y,y}(t)
 \underset{n\rightarrow +\infty}{\longrightarrow}E^{|R|}_{y,y}(\{m_{|R|}\}).
\]
Similarly,  we see
\[
 ||\left(\frac{A}{||A||}\right)^nx||^2 \underset{n\rightarrow +\infty}{\longrightarrow}E^A_{x,x}(\{||A||\}).
\]
%Finally, since $A$ is injective we observe that
%\[
 %||A^n(\varepsilon I + A)^{-n}y-y||^2=int_{0}^{||A||}|\frac{t^n}{(t+\varepsilon)^n}-1 |dE^A_{y,y}(t)
%\underset{n\rightarrow +\infty}{\longrightarrow}E^A_{y,y}(\{0\})=0.
%\]

According to the assumptions of (2), we must have at least one of the two spectral projections 
$E^{|R|}(\{m_{|R|}\}$ or $E^A(\{||A||\}$ null. Thus,
using the three previous facts, we see that $X_0=0$ which implies $X=0$. So, we get a contradiction. 
Consequently, using Proposition \ref{criterion}, it follows that the circle $\mathcal{C}(0,m_{|R|}/||A||)$ does 
not intersect $\Lambda_{int}(R,A\otimes S)$. From the two firsts steps of the proof, we derive that 
$\Lambda_{int}(R,A\otimes S)=\mathbb{D}(0,\frac{m_{|T|}}{||T||}]^c$.
%%Consequently, by using the Lemma \ref{sm}, if $\lambda\in\s{T}$, then $|\lambda|>\frac{m_A}{||A||}$.
This finishes the proof of Theorem \ref{intertwining values}.

\end{proof}

%Let $A$ be an injective positive operator acting on $\mathcal{B}(\mathfrak{L})$, and $r\in\mathbb{R}_+^*$. Consider
%\[
%\mathcal{A}_r=\{L\in\mathcal{B}(\mathfrak{L}) : \exists c\geq0\ such\ that\ \forall x\in\mathfrak{L}, \forall 
%n \in \mathbb{N}, ||r^{-n}A^nLx||\leq c||A^nx||\}.
%\]

\begin{coro}\label{extended spectrum}
Let $T$ be a pure quasinormal operator acting on a Hilbert space $\mathcal{H}$. We have
$\s{T}=\mathbb{D}(0,\frac{m_{|T|}}{||T||}[^c$ when $ m_{|T|}$ and $ ||T||$ are in
$\sigma_p(|T|)$,
and we have $ \s{T}=\mathbb{D}(0,\frac{m_{|T|}}{||T||}]^c $ when $ m_{|T|}$ and $||T||$ are not both in $\sigma_p(|T|)$.
\end{coro}
\begin{proof}
Applying Theorem \ref{brown} and taking into account Remark \ref{smbl} we see that $T$ is unitarily equivalent 
to the operator $A_T\otimes S$ acting on the Hilbert space 
$\mathfrak{L}_T\otimes H^2$, where $\mathfrak{L}_T=\mathcal{H}\ominus V_T\mathcal{H}$. We set for simplicity $A:=A_T$. 
We clearly have $m_{|T|}=m_A$, $||T||=||A||$ and $\sigma_p(|T|)=\sigma_p(A)$. 
Therefore, from now on, we may assume that $T$ is under the form $A\otimes S$ and $\mathcal{H}=\mathfrak{L}_T\otimes H^2$. 
Then, it suffices to apply Theorem \ref{intertwining values} with $R=A \otimes S$.
\end{proof}
%\begin{rmq}\in\mathcal{A}_{|\lambda|}
 %In virtue of the last corollary and its proof, we get easily the following equality
%\[
% \s{T}=\{\lambda\in\mathbb{C} : E^A(\sigma(T)\cap\sigma(\lambda T))\neq0\}.
%\]
%Such an equality which stays valid in the case of $A$ is non invertible.
%\end{rmq}

\end{section}

\begin{section}{Case of subnormal operators}
 It is known that an operator $T\in\lh$ is subnormal if there is a Hilbert space $\mathcal{K}$ containing $\mathcal{H}$
and a normal operator $N\in\mathcal{B}(\mathcal{K})$ such that $S=N|_\mathcal{H}$. This extension is minimal (m.n.e.) if
$\mathcal{K}$ has no proper subspace that reduces $N$ and contains $\mathcal{H}$. In addition, 
we know that every quasinormal operator
is subnormal. So, we show the following theorem in the more general case, that is the subnormal one. 
In particular, it is true for 
quasinormal operators.

\begin{thm}\label{subnl}
 Let $N\in\mathcal{B}(\mathcal{E})$ and $T\in\mathcal{B}(\mathcal{F})$ be normal 
 and pure subnormal operators respectively, 
such that the operator $R=N\oplus T\in\mathcal{B}(\mathcal{E}\oplus\mathcal{F})$ is injective.
Let $\lambda\in\s{Z}$, and let 
\[X=\begin{bmatrix}
X_1 & X_2 \\
X_3 & X_4 \\
\end{bmatrix}\in E_{ext}(Z,\lambda).
\]
Then $X_3=0$, $X_1\in E_{ext}(N,\lambda)$, $X_4\in E_{ext}(T,\lambda)$
and $X_2 \in E_{int}(N,T,\lambda)$.
\end{thm}
\begin{proof}
The hypothesis imply
 \begin{equation}\label{sys}
 \left\{
  \begin{array}{l l}
    NX_1=\lambda X_1N \\
    NX_2=\lambda X_2T \\
    TX_3=\lambda X_3N \\
    TX_4=\lambda X_4T \\
  \end{array} \right.
\end{equation}
Clearly, it suffices to show that $X_3=0$.
So, let 
\[M=\begin{bmatrix}
T & Y \\
0 & T_1 \\
\end{bmatrix}\in\mathcal{B}(\mathcal{F}\oplus\mathcal{G}),
\]
be the m.n.e. of $T$, and consider the following operators defined on $\mathcal{E}\oplus\mathcal{F}\oplus\mathcal{G}$
by
\[
\tilde{M}=\begin{bmatrix}
0 & 0 & 0 \\
0 & T & Y \\
0 & 0 & T_1 \\
\end{bmatrix},\  \tilde{N}=\begin{bmatrix}
N & 0 & 0 \\
0 & 0 & 0 \\
0 & 0 & 0 \\
\end{bmatrix},\  \tilde{X}=\begin{bmatrix}
0 & 0 & 0 \\
X_3 & 0& 0 \\
0 & 0 & 0 \\
\end{bmatrix}.
\]
Then formulas \ref{sys} imply $\tilde{M}\tilde{X}=\lambda\tilde{X}\tilde{N}$.
But both $\tilde{M}$ and $\tilde{N}$ are normal operators, by using the Fuglede-Putnam theorem it follows
$\tilde{M}^*\tilde{X}=\bar{\lambda}\tilde{X}\tilde{N}^*$.
From this, we have easily $T^*X_3=\bar{\lambda}X_3N^*$. Hence, for all $m,n\in\mathbb{N}$ we get the following system
\begin{equation*}
 \left\{
  \begin{array}{l l}
    T^nX_3=\lambda^nX_3N^n \\
    T^{*m}X_3=\bar{\lambda}^mX_3N^{*m}, \\
  \end{array} \right.
\end{equation*}
which implies, since $N$ is normal
\[
 T^{*m}T^nX_3=T^nT^{*m}X_3.
\]
Consequently
\[
 (T^{*m}T^n-T^nT^{*m})X_3=0,\  \ \forall m,n\in\mathbb{N},
\]
which means
\[
 Im(X_3)\subset\bigcap_{m,n\in\mathbb{N}}ker(T^{*m}T^n-T^nT^{*m}):=\mathcal{M}.
\]
Now, let $x\in\mathcal{M}$ then
\[
 T^{*m}T^n(Tx)=T^{*m}T^{n+1}x=T^{n+1}T^{*m}x=T^nT^{*m}(Tx).
\]
Thus $\mathcal{M}\in\mathcal{L}at(T)$. Furthermore, if $x\in\mathcal{M}$ then
\[
 T^{*m}T^n(T^*x)=T^{*m}(T^{n}T^*x)=T^{*m+1}T^{n}x=T^nT^{*m}(T^*x).
\]
Hence $\mathcal{M}\in\mathcal{L}at(T^*)$.
From this $\mathcal{M}$ is a reducing subspace for $T$. Therefore, there are two operators $M_1\in\mathcal{B}(\mathcal{M})$,
$M_1\in \mathcal{B}(\mathcal{M^\perp})$ such that $T=M_1\oplus M_2$. Moreover, in $\mathcal{M}\oplus\mathcal{M^\perp}$ 
the operators $TT^*$ and $T^*T$ have the following representations
\[
 TT^*=M_1M_1^*\oplus M_2M_2^*,\  \ T^*T=M_1^*M_1\oplus M_2^*M_2.
\]
Finally, let $x\in\mathcal{M}$, then $TT^*x=T^*Tx$ which implies $M_1M_1^*x=M_1^*M_1x$.
So $M_1$ is normal, and we get $\mathcal{M}=0$ since $T$ is pure. Consequently $X_3=0$
and the proof is complete. 
\end{proof}

\end{section}

\begin{section}{Extended eigenvalues and extended eigenspaces of quasi-normal operators}

%%By using the proof of Proposition \ref{criterion}, one can easily verify that if $\lambda\in\s{T}$, then
%%$\mathcal{A}_{|\lambda|}\neq\emptyset$.
The following theorem describes the spaces of extended 
eigenvectors of a pure quasi normal operator. We will use the notations introduced in Remark \ref{smbl}.
\begin{thm}\label{eigvect_pure_qn}
 Let $T$ be a pure quasinormal operator acting on a Hilbert space $\mathcal{H}$. 
Let $\lambda\in\s{T}$ then
\begin{align*}
 E_{ext}(T, \lambda)&=\mbox{weak}^{\ast}\mbox{-span}\{U_T^{\ast}(I\otimes S^m)diag(L, X_{1,1},...X_{n,n},...)U_T\\
&: m\in\mathbb{N}, L\in\mathcal{A}_{|\lambda|}(A_T)\},
\end{align*}
where $X_{n,n}$ is, for all $n\in\mathbb{N}$, the (unique bounded) extension on $\mathfrak{L}_T$ of the operator 
\[
\begin{array}{ccccc}
\breve{X}_{n,n} & : & ImA_T^n & \to & \mathfrak{L}_T \\
 & & A_T^nx & \mapsto & \lambda^{-n}A_T^nLx. \\
\end{array} 
\]
\end{thm}
\begin{proof}
As usual, we set $A:=A_T$, $\mathfrak{L}_T=\mathcal{H}\ominus V_T\mathcal{H}$ and 
$\mathcal{A}_{|\lambda|}:=\mathcal{A}_{|\lambda|}(A_T)$. 
Let $X_0\in\mathcal{B}(\mathcal{H})$ be a nonzero solution of $TX_0=\lambda X_0T$. Then, we have seen that 
$X_0=U_T^{\ast}XU_T$ where $X \in \mathcal{B}(\mathfrak{L}_T\otimes H^2)$ is 
solution of the equation $ (A\otimes S)X= \lambda X(A\otimes S)$.
Let $(A_{i,j})_{i,j\geq0}$ be the matrix of $A \otimes S$ in $\mathfrak{L}\otimes H^2$, i.e.,
\begin{equation*}
 A_{i,j} = \left\{
  \begin{array}{l l}
 A & \quad \text{if $i=j+1$}\\
 0 & \quad \text{otherwise.}\\
  \end{array} \right.
\end{equation*}
Consider for all $\alpha\in\overline{\mathbb{D}}$, the operator $J_\alpha$ whose the matrix in $\mathfrak{L}\otimes H^2$
is defined by
\begin{equation*}
 (J_\alpha)_{i,j} = \left\{
  \begin{array}{l l}
 \alpha^iI & \quad \text{if $i=j$}\\
 0 & \quad \text{otherwise.}\\
  \end{array} \right.
\end{equation*}
Then one can verify that $J_\alpha (A\otimes S)=\alpha (A\otimes S)J_\alpha$. In particular, $J_0 (A\otimes S)=0$.
Now let $\lambda \in \s{T}$ and let $X=(X_{i,j})$ be a nonzero operator 
acting on $\mathfrak{L}\otimes H^2$, and $\lambda\in\mathbb{C}$ (necessarily nonzero)
verifying $(A\otimes S)X=\lambda X(A\otimes S)$. A left composition by $J_0$ gives 
\[
 0=J_0 (A\otimes S)X=\lambda J_0X (A\otimes S)=J_0X (A\otimes S).
\]
But
\begin{equation*}
 (J_0X(A\otimes S))_{i,j} = \left\{
  \begin{array}{l l}
 X_{0,j+1}A & \quad \text{if $i=0$}\\
 0 & \quad \text{otherwise,}\\
  \end{array} \right.
\end{equation*}
which implies $X_{0,j+1}=0$ for all $j$, since $A$ has dense range.
In addition, we know that $(A\otimes S)X=\lambda X(A\otimes S)$ implies $(A\otimes S)^nX=\lambda^n X(A\otimes S)^n$ 
for all $n\in\mathbb{N}$. A same process
gives $X_{n,m}=0$ for all $0<n<m$. Consequently, $X$ has a lower triangular matrix.\\
Thus, if we denote by $X(m)$ the 
operator whose the matrix is 
\begin{equation*}
 (X(m))_{i,j} = \left\{
  \begin{array}{l l}
 X_{i,j} & \quad \text{if $j=m+i$}\\
 0 & \quad \text{otherwise,}\\
  \end{array} \right.
\end{equation*}
for all $m\in\mathbb{Z}$. 
We can prove that
\[
X=\underset{n\rightarrow +\infty}{weak^{\ast}lim} \left(\sum_{k=0}^{n}\left(1-\frac{k}{n+1}\right)X(-k)\right). 
\] 
Moreover, we observe that there exists an operator $Y$ acting on $\mathfrak{L}\otimes H^2$ such that
$X(-n)=(I\otimes S^n)(Y(0))$
for all $n\in\mathbb{N}$. Furthermore, one can verify that
$(A\otimes S)X(-n)=\lambda X(-n)(A\otimes S)$ if and only if 
$(A\otimes S)Y(0)=\lambda Y(0)(A\otimes S)$. Therefore, we are reduced to examine the case where $X=X(0)$.\\
We have 
\begin{equation*}
 ((A\otimes S)X(0))_{i,j} = \left\{
  \begin{array}{l l}
 AX_{i-1,i-1} & \quad \text{if $i=j+1$}\\
 0 & \quad \text{otherwise,}\\
  \end{array} \right.
\end{equation*}
and
\begin{equation*}
 (\lambda X(0)(A\otimes S))_{i,j} = \left\{
  \begin{array}{l l}
 \lambda X_{i,i}A & \quad \text{if $i=j+1$}\\
 0 & \quad \text{otherwise.}\\
  \end{array} \right.
\end{equation*}
Hence, for all $n\in\mathbb{N}$ we have $AX_{n,n}=\lambda X_{n+1,n+1}A$.
Thus, we get for all $n$, $\lambda^{-n}A^nX_{0,0}=X_{n,n}A^n$. On the one hand, 
since the range of $A$ is dense, it implies
that $X_{0,0}$ is necessarily non-null. On the other hand, we see that
\[
 |\lambda|^{-n}||A^nX_{0,0}x||\leq  |X||||A^nx||,\   \ \forall n\in\mathbb{N},\ \ \forall x\in\mathfrak{L}.
\]
Consequently, $X_{0,0}$ is a non-zero element of $\mathcal{A}_{|\lambda|}:=\mathcal{A}_{|\lambda|}(A)$.

Reciprocally, if $L\in\mathcal{A}_{|\lambda|}  \setminus \{0\}$, we define for all $n\in\mathbb{N}$, the operator  
\[
\begin{array}{ccccc}
\breve{X}_{n,n} & : & ImA^n & \to & \mathfrak{L} \\
 & & A^nx & \mapsto & \lambda^{-n}A^nLx. \\
\end{array} 
\]
Since $L\in\mathcal{A}_{|\lambda|}$, there is $c\geq0$ such that for all $y\in ImA^n$, 
$||\breve{X}_{n,n}y||\leq c||y||$. Also, $ImA^n$ is dense in $\mathfrak{L}$, thus $\breve{X}_{n,n}$
has a (unique bounded) extension on $\mathfrak{L}$, which will be denoted by $X_{n,n}$.
It remains to verify that $AX_{n,n}=\lambda X_{n+1,n+1}A$, for all $n\in\mathbb{N}$. So, let $x\in\mathfrak{L}$,
and $y=A^nx$, then
\begin{align*}
 AX_{n,n}y&=AX_{n,n}A^nx=A\breve{X}_{n,n}A^nx=\lambda^{-n}A^{n+1}Lx\\
 &=\lambda\lambda^{-n-1}A^{n+1}Lx=\lambda\breve{X}_{n+1,n+1}A^{n+1}x=\lambda X_{n+1,n+1}Ay,
\end{align*}
by density, we get $AX_{n,n}=\lambda X_{n+1,n+1}A$, which implies $X\in E_{ext}(T,\lambda) \setminus \{0\}$, as we wanted.

%%if $|\lambda|<1/||A^{-1}||||A||$ then $X_{0,0}=0$ which implies that $X_{n,n}=0$ for all $n$.
%%So, $\lambda$ must be in $\mathbb{D}(0,\frac{m_A}{||A||}[^c$ and the proof is complete.

%Let $R$ be a self-adjoint operator acting on a Hilbert space $H$, we denote by $m_R$ the real number defined by 
%\[
%m_R=\inf\{<Rx,x> : ||x||=1\}. 
% \]
%We observe that $m_R=1/||R^{-1}||$ when $R$ is an invertible positive operator. 

\end{proof}
\begin{rmq}
 Let $A$ be an injective positive operator on a Hilbert space $\mathfrak{L}$ and let $|\lambda|\leq1$, 
then we can easily verify that 
$\mathcal{A}_{|\lambda|}:=\mathcal{A}_{|\lambda|}(A)$ is an algebra, that which is not true in general, when $|\lambda|>1$
(see Example (\ref{contex})).
%Moreover, in the case of the positive operator $A$ is invertible, the set $\mathcal{A}_{|\lambda|}$ is defined by
%\[
% \mathcal{A}_{|\lambda|}=\{L\in\mathcal{B}(\mathfrak{L}) : \sup_{n\in\mathbb{N}}||\lambda^{-n}A^nLA^{-n}||<+\infty\}.
%\]
Recall that $\mathcal{A}_{1}$ is the Deddens algebra given in \cite{dedalg}.
Finally, If we denote by $D_{A,L,\lambda}$ the operator defined on $\mathfrak{L}\otimes H^2$
by 
\begin{equation*}
 (D_{A,L,\lambda})_{i,j\in\mathbb{N}} = \left\{
  \begin{array}{l l}
 \lambda^{-i}A^iLA^{-i}& \quad \text{if $i=j$}\\
 0 & \quad \text{otherwise.}\\
  \end{array} \right.
\end{equation*}
Then, we get the following corollary.
\end{rmq}
\begin{coro}\label{eigvec2}
 Let $A$ be an invertible positive operator on a Hilbert space $\mathfrak{L}$, and $T=A\otimes S$. 
If $\lambda\in\s{T}$ then
\[
E_{ext}(T,\lambda)=\mbox{weak}^{\ast}\mbox{-span}\{(I\otimes S^n) D_{A,L,\lambda} : n\in\mathbb{N}, L\in\mathcal{A}_{|\lambda|}(A)\}. 
\]
\end{coro}
%\begin{proof}
 %Let $X\in\mathcal{B}(\mathfrak{L})$ a nonzero solution of $TX=\lambda XT$. As in the proof of Proposition \ref{criterion}, we get
%$X=\sum_{n\in\mathbb{N}}X(-n)$, and for all $n\in\mathbb{N}$, there exists an operator $Y$ such that
%\[
 %X(-n)=(I\otimes S^n)(Y(0)).
%\]
%Thus, we reduce to the case of $X=X(0)$. Moreover, we got 
%\begin{equation}\label{solt}
 %X=diag(X_{0,0},\lambda^{-1}AX_{0,0}A, ...,\lambda^{-n} A^nX_{0,0}A^{-n}, ...).
%\end{equation}
%In particular
%\[
 %\sup_{n\in\mathbb{N}}||\lambda^{-n}A^nLA^{-n}||=||X||<+\infty
%\]
%\end{proof}

\begin{ex}
 Let $T=A\otimes S$ such that
\[
 A= \begin{bmatrix}
\alpha & 0 \\
0 & \beta \\
\end{bmatrix}
,\  \  \alpha>\beta>0, 
\]
and let 
\[
 L= \begin{bmatrix}
a & b \\
c & d \\
\end{bmatrix}
,\  \  a,b,c,d\in\mathbb{C}.
\]
Then, if $\lambda\in\mathbb{C}^*$ we get
\[
 \lambda^{-n} A^nLA^{-n}= \begin{bmatrix}
\lambda^{-n}a & (\frac{\alpha}{\lambda\beta})^nb \\
(\frac{\beta}{\lambda\alpha})^nc & \lambda^{-n}d \\
\end{bmatrix}
. 
\]
So we distinguish the following cases :
\begin{enumerate}
 \item if $|\lambda|\geq\frac{\alpha}{\beta}>1$, 
 then $\mathcal{A}_{|\lambda|}=\mathcal{B}(\mathfrak{L})=\mathcal{M}_2(\mathbb{C})$.
 \item if $1\leq|\lambda|<\frac{\alpha}{\beta}$, then 
 \[
 \mathcal{A}_{|\lambda|}=\{\begin{bmatrix}
a & 0 \\
c & d \\
\end{bmatrix}
,\  \  a,c,d\in\mathbb{C}\}.
\]
 \item if $\frac{\beta}{\alpha}\leq|\lambda|<1$, then 
 \[
 \mathcal{A}_{|\lambda|}=\{\begin{bmatrix}
0 & 0 \\
c & 0 \\
\end{bmatrix}
,\  \  c\in\mathbb{C}\}.
\]
 \item if $|\lambda|<\frac{\beta}{\alpha}$, then $\mathcal{A}_{|\lambda|}=\{0\}$.
\end{enumerate}

\end{ex}

\begin{ex}\label{contex}
 Let $T=A\otimes S$ such that
\[
 A= \begin{bmatrix}
\alpha & 0 & 0 \\
0 & \beta & 0 \\
0 & 0 & \gamma \\
\end{bmatrix}
,\  \  \alpha>\beta>\gamma>0, 
\]
and let 
\[
 L_1= \begin{bmatrix}
0 & 1 & 0 \\
0 & 0 & 0 \\
0 & 0 & 0 \\
\end{bmatrix}
,\  \  L_2= \begin{bmatrix}
0 & 0 & 0 \\
0 & 0 & 1 \\
0 & 0 & 0 \\
\end{bmatrix}.
\]
Then, one can verify that
\[
||\lambda^{-n}A^nL_1A^{-n}||=|\lambda|^{-n}(\frac{\alpha}{\beta})^n,
\]
\[
||\lambda^{-n}A^nL_2A^{-n}||=|\lambda|^{-n}(\frac{\beta}{\gamma})^n,
\]
and
\[
||\lambda^{-n}A^nL_1L_2A^{-n}||=|\lambda|^{-n}(\frac{\alpha}{\gamma})^n.
\]
Now, let $\alpha,\beta,\gamma$ and $\lambda$ be such that 
\[
\frac{\alpha}{\beta}=\frac{\beta}{\gamma}=|\lambda|,
\]
then, clearly $L_1,L_2\in\mathcal{A}_{|\lambda|}$, but $L_1L_2\notin\mathcal{A}_{|\lambda|}$.
\end{ex}

In the next result, we describe completely the extended eigenspaces of a general quasinormal operator.
\begin{thm}\label{eigvect_qn}
 Let $R$ be an injective quasinormal operator on a Hilbert space $\tilde{\mathcal{H}}$, and consider
 $R=N\oplus T \in\mathcal{B}(\mathcal{E}\oplus\mathcal{H})$ the canonical decomposition 
 of $R$ into a direct sum of a normal
 operator $N \in \mathcal{B}(\mathcal{E})$ and a pure quasinormal operator $T \in\mathcal{B}(\mathcal{H})$.
 Let $\lambda\in\s{R}$ then $E_{ext}(R,\lambda)$ is the following subspace of $\mathcal{B}(\mathcal{E}\oplus\mathcal{H})$
\[
\{\begin{bmatrix}
U &VU_T \\
0 & W\\
\end{bmatrix}: U \in E_{ext}(N,\lambda), V \in E_{int}(N,A_T\otimes S,\lambda), W \in E_{ext}(T,\lambda)\}
\]
where $E_{int}(N,A_T \otimes S,\lambda)$ is the set of operators 
$V \in \mathcal{B}( \oplus_{k=0}^{+\infty} \mathfrak{L}_T,\mathcal{E})$ 
whose matricial form are given by
$V=[V_0,\ldots,V_n,\ldots]$, 
where $V_{0} \in \mathcal{A}_{|\lambda|}(N,A_T)$ and
 $V_{n}$ is, for all $n\in\mathbb{N^{\ast}}$, the (unique bounded) extension on $\mathfrak{L}$ of the operator 
\[
\begin{array}{ccccc}
\breve{V}_{n} & : & ImA^n & \to & \mathcal{E} \\
 & & A^nx & \mapsto & \lambda^{-n}N^n V_{0}x. \\
\end{array} 
\]
\end{thm}
\begin{proof}
Let $X$ be an extended eigenvector of $R$ associated with the extended eigenvalue $\lambda$. According to 
Theorem \ref{subnl} it suffices to describe the upper off-diagonal 
coefficient $X_2$ in the matrix representation of $X$ with respect to the orthogonal direct sum 
$\tilde{\mathcal{H}}= \mathcal{E}\oplus\mathcal{H}$. Clearly, we have $X_2=VU_T$ where 
$V=[V_0,\ldots,V_n,\ldots] \in E_{int}(N,A_T \otimes S,\lambda)$. 
For convenience, we write $A=A_T$. We see that
$N^nV_0=\lambda^{n}V_nA^n$ for every $n\in \mathbb{N}$. Thus, we have $||N^nV_0x||\leq ||V|||\lambda|^n||A^nx||$ and
hence $V_{0} \in \mathcal{A}_{|\lambda|}(N,A)$. 

Conversely, by using assumptions, we get $NV=\lambda A\otimes S$ and any matrix of the form
\[
X=\begin{bmatrix}
U &VU_T \\
0 & W\\
\end{bmatrix},
\]
where $ U \in E_{ext}(N,\lambda)$ and $ W \in E_{ext}(T,\lambda)\}$, 
is an extended eigenvector of $R$. It ends the proof.
\end{proof}

We can now describe the extended spectrum of a general quasinormal operator.
\begin{coro}\label{extspect_qn}
 Let $R$, $N$, $T$ and $\tilde{\mathcal{H}}$ be as in the last theorem, then 
 we have
\[
 \s{R}=\s{N\oplus T}=\s{N} \cup \mathbb{D}(0,\frac{m_{|N|}\wedge m_{|T|}}{||T||}[^c
\]
if one of the following assumptions holds:\\
- a) $m_{|N|}<m_{|T|}$ and $(m_{|N|},||T||) \in \sigma_p(|N|)\times \sigma_p(|T|)$;\\
- b) $m_{|N|}=m_{|T|}$, and $(m_{|N|},||T||) \in \sigma_p(|N|)\times \sigma_p(T)$ 
 or $(m_{|T|},||T||) \in \sigma_p(|T|)^2$;\\
- c) $m_{|N|}>m_{|T|}$ and $(m_{|T|},||T||) \in \sigma_p(|T|)^2.$\\
Else we have
\[
 \s{R}=\s{N} \cup \mathbb{D}(0,\frac{m_{|N|}\wedge m_{|T|}}{||T||}]^c.
\]
\end{coro}
\begin{proof}
 Using Theorem \ref{eigvect_qn}, we see that 
\[
 \s{R}=\s{N}\cup \s{T} \cup \Lambda_{int}(N,A_T \otimes S). 
\]
%%Using similar arguments as in the proof of Theorem \ref{eigvect_qn}, we show that 
%%$\Lambda_\lambda(N,A_T\otimes S)=\{\lambda: \mathcal{A}_{|\lambda|}(N,A) \neq \{0\} \}$.
Taking into account Corollary \ref{extended spectrum}, we see that the proof rests on an application of Theorem \ref{intertwining values}.
\end{proof}

\end{section}
\begin{section}{Lifting of eigenvectors of pure quasi-normal operators}
 In \cite[Theorems 1 and 3]{yosh}, the author gives a necessary and sufficient 
 condition that an operator commuting with a quasinormal operator 
have an extension commuting with the normal extension of the quasinormal operator. 
In Theorem \ref{gener} we generalize this to 
operators intertwining two quasinormal operators. 
First we introduce the following proposition (see \cite{embry} for the proof).
\begin{prop}\label{cnas}
 Let $T_i\in\mathcal{B}(\mathcal{H}_i)$ be subnormal operator with m.n.e. $N_i\in\mathcal{B}(\mathcal{K}_i)$, $i=1,2$, 
and let $X\in\mathcal{B}(\mathcal{H}_2,\mathcal{H}_1)$. The following are equivalent :
\begin{enumerate}
 \item $X$ has a (unique) extension $\hat{X}\in\mathcal{B}(\mathcal{K}_2,\mathcal{K}_1)$ 
 such that $N_1\hat{X}=\hat{X}N_2$.
 \item There exists a constant $c\geq0$ such that
\[
 \sum_{i,j=0}^n<T_1^iXx_j,T_1^jXx_i>\leq c\sum_{i,j=0}^n<T_2^ix_j,T_2^jx_i>
\]
for all finite set $\{x_0,...,x_n\}$ in $\mathcal{H}_2$. 
\end{enumerate}
Moreover, if (2) holds, then $T_1X=XT_2$.
\end{prop}
We also need the following auxiliary lemma (see \cite{yosh} for the proof).
\begin{lm}\label{pd}
 Let $T\in\mathcal{B}(\mathcal{H})$ be injective quasinormal operator with the polar decomposition $T=V|T|$, and let 
$N\in\mathcal{B}(\mathcal{K})$ be the m.n.e. of $T$ with the polar decomposition $N=U|N|$. Then $U$ is unitary and
\[
 V=U|_\mathcal{H}\   \ and\   \ |T|=|N||_\mathcal{H}.
\]
\end{lm}
Now we prove the main theorem of this section.
\begin{thm}\label{gener}
 Let $T_i\in\mathcal{B}(\mathcal{H}_i)$ be injective quasinormal operator 
 with the polar decomposition $T_i=V_i|T_i|$ and let 
$N_i\in\mathcal{B}(\mathcal{K}_i)$, be the m.n.e. of $T_i$ with the polar decomposition $N_i=U_i|N_i|$, $i=1,2$. 
If $X\in\mathcal{B}(\mathcal{H}_2,\mathcal{H}_1)$, then the following are equivalent :
\begin{enumerate}
 \item $X$ has a (unique) extension $\hat{X}\in\mathcal{B}(\mathcal{K}_2,\mathcal{K}_1)$ such that $N_1\hat{X}=\hat{X}N_2$.
 \item $V_1 X=XV_2$ and $|T_1|X=X|T_2|$.
\end{enumerate}
\end{thm}
\begin{proof}
$(1)\Rightarrow (2)$. Let $\hat{X}\in\mathcal{B}(\mathcal{K}_2,\mathcal{K}_1)$ be an extension of $X$ such that $N_1\hat{X}=\hat{X}N_2$.
Then by using Fuglede-Putnam theorem, we get $N_1^*\hat{X}=\hat{X}N_2^*$.
Hence, we easily get  
\[
 |N_1|\hat{X}=\hat{X}|N_2|\ \ and\ \ U_1\hat{X}=\hat{X}U_2.
\]
By Lemma \ref{pd}, for all $x\in\mathcal{H}_2$
\[
 V_1 Xx=V_1\hat{X}x=U_1\hat{X}x=\hat{X}U_2x=\hat{X}V_2x=XV_2x,
\]
and 
\[
 |T_1|Xx=|T_1|\hat{X}x=|N_1|\hat{X}x=\hat{X}|N_2|x=\hat{X}|T_2|x=X|T_2|x.
\]
$(2)\Rightarrow (1)$. Let $U_i\in\mathcal{B}(\mathcal{K}_i^\prime)$ be the minimal unitary extension of $V_i$. Then for any finite set 
$\{x_0,...,x_n\}$ in $\mathcal{H}_2$
\begin{align*}
 \sum_{i,j=0}^n&<T_1^iXx_j,T_1^jXx_i>=\sum_{i,j=0}^n<V_1^i|T_1|^iXx_j,V_1^j|T_1|^jXx_i>\\
 =&\sum_{i,j=0}^n<V_1^iX|T_2|^jx_j,V_1^jX|T_2|^ix_i>=\sum_{i,j=0}^n<U_1^iX|T_2|^jx_j,U_1^jX|T_2|^ix_i>\\
 =&\sum_{i,j=0}^n<U_1^{*j}X|T_2|^jx_j,U_1^{*i}X|T_2|^ix_i>=||\sum_{i=0}^nU_1^{*i}X|T_2|^ix_i||_{\mathcal{K}_1^\prime}^2.
\end{align*}
Since for all $k\geq0$
\begin{align*}
 U_1^{*i}X|T_2|^ix_i&=U_1^{*i+k}U_1^kX|T_2|^ix_i=U_1^{*i+k}XV_2^k|T_2|^ix_i\\
 &=U_1^{*i+k}XU_2^k|T_2|^ix_i=U_1^{*i+k}XU_2^{i+k}U_2^{*i}|T_2|^ix_i,
\end{align*}
for all $i=0,...,n$, we have, by choosing $k$ such that $i+k=n$ for each $i$
\begin{align*}
 ||\sum_{i=0}^n&U_1^{*i}X|T_2|^ix_i||_{\mathcal{K}_1^\prime}^2=
 ||U_1^{*n}XU_2^{n}\sum_{i=0}^nU_2^{*i}|T_2|^ix_i||_{\mathcal{K}_1^\prime}^2\\
 =&||XU_2^{n}\sum_{i=0}^nU_2^{*i}|T_2|^ix_i||_{\mathcal{K}_1^\prime}^2=
 ||XU_2^{n}\sum_{i=0}^nU_2^{*i}|T_2|^ix_i||_{\mathcal{H}_1}^2\\
 \leq& ||X||_{\mathcal{H}_1}^2||\sum_{i=0}^nU_2^{*i}|T_2|^ix_i||_{\mathcal{K}^{\prime}_1}^2
 =||X||_{\mathcal{H}_1}^2\sum_{i,j=0}^n<U_2^i|T_2|^jx_j,U_2^j|T_2|^ix_i>\\
 =&||X||_{\mathcal{H}_1}^2\sum_{i,j=0}^n<|T_2|^jV_2^ix_j,|T_2|^iV_2^jx_i>
 =||X||_{\mathcal{H}_1}^2\sum_{i,j=0}^n<T_2^ix_j,T_2^jx_i>
\end{align*}
Hence, Proposition \ref{cnas} implies the first assertion. The proof is now complete.
\end{proof}
Now, let $A$ be a positive operator, and denote by $U$ the bilateral shift, 
then $A\otimes U$ is the m.n.e. of $A\otimes S$, and the last theorem implies the following corollary.
\begin{coro}\label{equiv}
  Let $A$ be an injective positive operator on a Hilbert space $\mathfrak{L}$, and let $X$ 
  be a bounded operator on $\mathfrak{L}\otimes H^2$, then the following are equivalent
\begin{enumerate}
 \item $X$ has a (unique) extension $\hat{X}\in\mathcal{B}(\mathfrak{L}\otimes L^2)$ 
 such that\\ $(A\otimes U)\hat{X}=\lambda\hat{X}(A\otimes U)$.
 \item $(A\otimes I)X=|\lambda|X(A\otimes I)$ and $(I\otimes S)X=\lambda/|\lambda| X(I\otimes S)$. 
\end{enumerate}
\end{coro}
Now, we use similar arguments from the proofs of Theorem \ref{intertwining values} 
and Theorem \ref{eigvect_pure_qn} to establish the following result.
\begin{thm}\label{ext-spectrum-n}
 Let $A$ be an invertible positive operator on a Hilbert space $\mathfrak{L}$, and $N=A\otimes U$. Denote by
$a:=\|A\|\|A^{-1}\|$, then if $ (||A||,||A^{-1}||^{-1})\in \sigma_p(A)^2$ we have
\[
\s{N}=\{z\in\mathbb{C} : \frac{1}{a}\leq |z|\leq a\}, 
\]
else we have
\[
\s{N}=\{z\in\mathbb{C} : \frac{1}{a} < |z| < a\}. 
\]
Moreover, if $\lambda\in\s{N}$ then
\[
E_{ext}(N,\lambda)=\mbox{weak}^{\ast}\mbox{-span}\{(I\otimes U^m) \hat{D}_{A,L,\lambda} : m\in\mathbb{Z}, L\in\hat{\mathcal{A}}_{|\lambda|}\}. 
\]
where $\hat{D}_{A,L,\lambda}$ is the operator defined by 
\begin{equation*}
 (\hat{D}_{A,L,\lambda})_{i,j\in\mathbb{Z}} = \left\{
  \begin{array}{l l}
 \lambda^{-i}A^iLA^{-i}& \quad \text{if $i=j$}\\
 0 & \quad \text{otherwise.}\\
 \end{array} 
\right.
\end{equation*}
and
\[
 \hat{\mathcal{A}}_{|\lambda|}=\{L\in\mathcal{B}(\mathfrak{L}) : \sup_{n\in\mathbb{Z}}||\lambda^{-n}A^nLA^{-n}||<+\infty\}.
\]

\end{thm}
Indeed, if $ (||A||,||A^{-1}||^{-1})\in \sigma_p(A)^2$ it suffices to consider $L_1=u \otimes v$ and $L_2=v \otimes u$ 
where $u$ and $v$ are
eigenvectors of $A$ associated with $||A||$ and $||A^{-1}||^{-1}$ respectively.\\ 
Else, we use the inequality
\[
 |<Lx,y>|
  \leq C||\left(\frac{A}{||A||}\right)^m x||||\left(\frac{A^{-1}}{||A^{-1}||}\right)^m y||
\]
(which is available for any $L \in \hat{\mathcal{A}}_{a}$ and any $m \in \mathbb{N}$), in order to show that $L$ is necessarily
null (see proof of Theorem \ref{intertwining values}). We proceed similarly for proving that
$\hat{\mathcal{A}}_{a^{-1}}=\{0\}$.
\begin{thm}
 Let $A$ be an invertible positive operator on a Hilbert space $\mathfrak{L}$, $T=A\otimes S$ 
 and $N=A\otimes U$. Let $\lambda \in \s{N}$ and $X\in \mathcal{B}(\mathfrak{L}\otimes H^2)$, then $X$ has a (unique) 
extension $\hat{X}\in\mathcal{B}(\mathfrak{L}\otimes L^2)$ such that
\[
 N\hat{X}=\lambda\hat{X}N,
\] 
if and only if 
\[
X\in \mbox{weak}^{\ast}\mbox{-span}\{(I\otimes S^n) D_{A,L,\lambda} : n\in\mathbb{N}, L\in E_{ext}(A,|\lambda|)\}. 
\]
\end{thm}
\begin{proof}
Let $\hat{X}=(\hat{X}_{i,j})_{i,j\in\mathbb{Z}}$ be an operator acting on $\mathfrak{L}\otimes L^2$ 
such that
\begin{equation}\label{eq1}
 N\hat{X}=\lambda\hat{X}N,
\end{equation}
Consider for all $\alpha\in\mathbb{T}$, the operator $\hat{J}_\alpha$ whose the matrix in $\mathfrak{L}\otimes L^2$
is defined by
\begin{equation*}
 (\hat{J}_\alpha)_{i,j} = \left\{
  \begin{array}{l l}
 \alpha^iI & \quad \text{if $i=j$}\\
 0 & \quad \text{otherwise.}\\
  \end{array} \right.
\end{equation*}
Then one can verify that $\hat{J}_\alpha N=\alpha N\hat{J}_\alpha$. Hence, let $\alpha,\beta\in\mathbb{T}$, 
if we apply to both 
sides of the Equation (\ref{eq1}) the operator $\hat{J}_\alpha$ from the left, 
and the operator $\hat{J}_\beta$ from the right, we get
\[
N\hat{J}_\alpha\hat{X}\hat{J}_\beta =\frac{\lambda}{\alpha\beta} \hat{J}_\alpha\hat{X}\hat{J}_\beta N.
\]
Now, let $m\in\mathbb{Z}$, $\theta\in[0,2\pi]$ and put $\alpha=\beta^{-1}=e^{i\theta}$, then the last equation implies
\[
N\int_0^{2\pi}e^{-im\theta}\hat{J}_{e^{i\theta}}\hat{X}\hat{J}_{e^{-i\theta}}dm(\theta)=
\lambda\int_0^{2\pi}e^{-im\theta}\hat{J}_{e^{i\theta}}\hat{X}\hat{J}_{e^{-i\theta}} dm(\theta) N,
\]
where the integrals are well defined in Bochner sense.
Hence
\[
N\hat{X}(m)=\lambda\hat{X}(m) N.
\]
where $\hat{X}(m)$ is the operator acting on $\mathfrak{L}\otimes L^2$ whose the matrix is given by 
\begin{equation*}
 (\hat{X}(m))_{i,j} = \left\{
  \begin{array}{l l}
 \hat{X}_{i,j} & \quad \text{if $i=m+j$}\\
 0 & \quad \text{otherwise,}\\
  \end{array} \right.
\end{equation*}
On the other hand, one can easily verify that $\hat{J}_{e^{i\theta}}\hat{X}\hat{J}_{e^{-i\theta}}$ is an extension of 
the operator $J_{e^{i\theta}}XJ_{e^{-i\theta}}$, so that $\hat{X}(m)$ 
is an extension of $X(m)=(I\otimes S^m) D_{A,L,\lambda}$.
Also, by using the last theorem, there exists $L\in\hat{\mathcal{A}}_{|\lambda|}$ such that
\[
\hat{X}(m)=(I\otimes U^m) \hat{D}_{A,L,\lambda}.
\]
Now, suppose that $m<0$. In this case, $L=0$. Indeed, if $L\neq 0$, then
$\mathfrak{L}\otimes H^2\notin\mathcal{L}at(\hat{X}(m))$, which means that there is 
no bounded operator on $\mathfrak{L}\otimes H^2$ for which $\hat{X}(m)$ is an extension.
Now assume that $m\geq0$. Then $\hat{X}(m)$ is an extension of the operator
$X(m)$, and by using the Corollary \ref{equiv}, we have an equivalence with the two following equations
\[
(A\otimes I)X=|\lambda|X(A\otimes I)\ \ and\ \ (I\otimes S)X=\lambda/|\lambda| X(I\otimes S).
\]
One can easily verify that the last equalities are equivalent to
\[
AL=|\lambda|LA.
\]
which means that $L\in E_{ext}(A,|\lambda|)$. 
The converse is easy and will be left to the reader. 
\end{proof}

\begin{rmq}
Let $A$ be an invertible positive operator on a Hilbert space $\mathfrak{L}$ such that
$ (||A||,||A^{-1}||^{-1})\in \sigma_p(A)^2$, $T=A\otimes S$ and $N=A\otimes U$ 
the m.n.e. of $T$. As a direct result of the last theorem, we can summarize the relationship between 
extended eigenvectors of $T$ and $N$ in the four following cases :\\
%\begin{enumerate}
1) If $|\lambda|\in[1/a,a]$ and let 
\[
X=(I\otimes S^n) D_{A,L,\lambda}, \  \ n\in\mathbb{N}.
\]
Suppose that $L\in E_{ext}(A,|\lambda|)$, then $X$ has a (unique) extension $\hat{X}\in E_{ext}(N,\lambda)$.\\
2) With the same hypotheses, if we suppose that $L\notin E_{ext}(A,|\lambda|)$, then $X$ doesn't have 
any extension in $E_{ext}(N,\lambda)$.\\
3) Let $|\lambda|\in[1/a,a]$ and $\hat{X}\in E_{ext}(N,\lambda)\backslash \lbrace 0 \rbrace$ be such that
\[
\hat{X}=(I\otimes U^m) \hat{D}_{A,L,\lambda},\   \ m<0,
\]
then there is no bounded operator on $\mathfrak{L}\otimes H^2$ for which $\hat{X}$ is an extension.\\
4) If $|\lambda|>a$, and let
\[
X=(I\otimes S^n) D_{A,L,\lambda}, \ \ n\in\mathbb{N},\ \ L\in\mathcal{A}_{|\lambda|}\backslash \lbrace 0 \rbrace,
\]
then $X\in E_{ext}(T,\lambda)$, but it has no extension in $E_{ext}(N,\lambda)$.\\
When $ (||A||,||A^{-1}||^{-1})\notin \sigma_p(A)^2$, the reader can adapt this remark by using 
Theorem \ref{ext-spectrum-n} and conclude.
%\end{enumerate}
\end{rmq}
\end{section}

\bibliographystyle{abbrv}
%\bibliographystyle{plain}
%\bibliography{ref2}

\begin{thebibliography}{}

\end{thebibliography}


\begin{thebibliography}{10}
 
\bibitem{hsn}
H.~Alkanjo.
\newblock On extended eigenvalues and extended eigenvectors of truncated shift.
\newblock {\em Concrete Operators}, 1:19--27, 2013.

\bibitem{volt1}
A.~Biswas, A.~Lambert, and S.~Petrovic.
\newblock Extended eigenvalues and the {V}olterra operator.
\newblock {\em Glasg. Math. J.}, 44(3):521--534, 2002.

\bibitem{brownq}
A.~Brown.
\newblock On a class of operators.
\newblock {\em Proc. Amer. Math. Soc.}, 4:723--728, 1953.

\bibitem{brown}
S.~Brown.
\newblock Connections between an operator and a compact operator that yield
  hyperinvariant subspaces.
\newblock {\em J. Operator Theory}, 1(1):117--121, 1979.

\bibitem{casalk}
G.~Cassier and H.~Alkanjo.
\newblock Extended spectrum, extended eigenspaces and normal operators.
\newblock {\em J. Math. Anal. Appl}, 418(1):305--316, 2014.

\bibitem{dedalg}
J.~A. Deddens.
\newblock Another description of nest algebras.
\newblock In {\em Hilbert space operators ({P}roc. {C}onf., {C}alif. {S}tate
  {U}niv., {L}ong {B}each, {C}alif., 1977)}, volume 693 of {\em Lecture Notes
  in Math.}, pages 77--86. Springer, Berlin, 1978.

\bibitem{embry}
M.~Embry-Wardrop.
\newblock Quasinormal extensions of subnormal operators.
\newblock {\em Houston Journal of mathematics, No. (2)}, 7:191--204, 1981.

\bibitem{kar}
M.~T. Karaev.
\newblock On extended eigenvalues and extended eigenvectors of some operator
  classes.
\newblock {\em Proc. Amer. Math. Soc.}, 134(8):2383--2392 (electronic), 2006.

\bibitem{prcy}
H.~W. Kim, R.~Moore, and C.~M. Pearcy.
\newblock A variation of {L}omonosov's theorem.
\newblock {\em J. Operator Theory}, 2(1):131--140, 1979.

\bibitem{lamb}
A.~Lambert.
\newblock Hyperinvariant subspaces and extended eigenvalues.
\newblock {\em New York J. Math.}, 10:83--88, 2004.

\bibitem{lmnsv}
V.~Lomonosov.
\newblock Invariant subspaces for operators commuting with compact operators.
\newblock {\em J. Operator Theory}, 7(1):213--214, 1973.

\bibitem{shkarin}
S.~Shkarin.
\newblock Compact operators.
\newblock {\em Journal of Mathematical analysis and applications}, 2007.

\bibitem{yosh}
T.~Yoshino.
\newblock On the commuting extensions of nearly normal operators.
\newblock {\em T\^ohoku Math. J. (2)}, 25:263--272, 1973.

 
 
\end{thebibliography}

\end{document}